\documentclass[12pt,twoside]{amsart}
\usepackage{amsmath, amsthm, amscd, amsfonts, amssymb, graphicx}
\usepackage{enumerate}
\usepackage[colorlinks=true,
linkcolor=blue,
urlcolor=cyan,
citecolor=red]{hyperref}
\usepackage{mathrsfs}
\newtheorem{theorem}{Theorem}[section]
\newtheorem{lemma}{Lemma}[section]
\newtheorem{remark}{Remark}[section]

\newtheorem{corollary}{Corollary}[section]
\newtheorem{example}{Example}[section]
\newtheorem{proposition}{Proposition}[section]
\numberwithin{equation}{section}

\begin{document}
\title{Advanced Refinements of Numerical Radius Inequalities}
\author{Farzaneh Pouladi Najafabadi and Hamid Reza Moradi}

\begin{abstract}
We prove several numerical radius inequalities for linear operators in Hilbert spaces. It is shown, among other inequalities, that if $A$ is a bounded linear operator on a complex Hilbert space, then
\[\omega \left( A \right)\le \frac{1}{2}\sqrt{\left\| {{\left| A \right|}^{2}}+{{\left| {{A}^{*}} \right|}^{2}} \right\|+\left\| \left| A \right|\left| {{A}^{*}} \right|+\left| {{A}^{*}} \right|\left| A \right| \right\|},\]
where $\omega \left( A \right)$, $\left\| A \right\|$, and $\left| A \right|$ are the numerical radius, the usual operator norm, and  the absolute value of  $A$, respectively. This inequality provides a refinement of an earlier numerical radius inequality due to Kittaneh, namely,
\[\omega \left( A \right)\le \frac{1}{2}\left( \left\| A \right\|+{{\left\| {{A}^{2}} \right\|}^{\frac{1}{2}}} \right).\]
Some related inequalities are also discussed.
\end{abstract}
\maketitle
%------------------------------------------------------------------------------------%
\pagestyle{myheadings}
\markboth{\centerline {Advanced Refinements of Numerical Radius Inequalities}}
{\centerline {F. P. Najafabadi and H. R. Moradi}}
\bigskip
\bigskip

\section{Introduction}
In this section, we introduce the notations and necessary prerequisites.
Let $\mathbb{B}\left( \mathscr{H} \right)$ denote the ${{C}^{*}}$-algebra of all bounded linear operators on a complex Hilbert space $\mathscr{H}$ with inner product $\left\langle \cdot,\cdot \right\rangle $. For $A\in \mathbb{B}\left( \mathscr{H} \right)$, let $\omega \left( A \right)$ and $\left\| A \right\|$ denote the numerical radius and the operator norm of $A$, respectively. Recall that $\omega \left( A \right)=\underset{\left\| x \right\|=1}{\mathop{\underset{x\in \mathscr{H}}{\mathop{\sup }}\,}}\,\left| \left\langle Ax,x \right\rangle  \right|$ and $\left\| A \right\|=\underset{\left\| x \right\|=1}{\mathop{\underset{x\in \mathscr{H}}{\mathop{\sup }}\,}}\,\left\| Ax \right\|$. We remark here that if $A\in \mathbb{B}\left( \mathscr{H} \right)$, and if $f$ is a non-negative increasing function on $\left[ 0,\infty  \right)$,
then $\left\| f\left( \left| A \right| \right) \right\|=f\left( \left\| A \right\| \right)$. Here $\left| A \right|$ stands for the positive operator
${{\left( {{A}^{*}}A \right)}^{\frac{1}{2}}}$. \\
It is clear that $\omega \left( \cdot \right)$ defines a norm
on $\mathbb{B}\left( \mathscr{H} \right)$, which is equivalent to the operator norm $\left\| \cdot \right\|$. In fact, for every $A\in \mathbb{B}\left( \mathscr{H} \right)$,
\begin{equation}\label{0}
\frac{1}{2}\left\| A \right\|\le \omega \left( A \right)\le \left\| A \right\|.
\end{equation}
The inequalities in \eqref{0} are sharp. The first inequality becomes an equality if ${{A}^{2}}=0$. The second inequality becomes an equality if $A$ is normal.

An important inequality for $\omega \left( A \right)$ is the power inequality \cite{berger} stating that
\begin{equation}\label{06}
\omega \left( {{A}^{n}} \right)\le {{\omega }^{n}}\left( A \right),
\end{equation}
for $n=1,2,\ldots $. 

In \cite{5}, Kittaneh improved the second inequality in \eqref{0}, and obtained the following result:
\begin{equation}\label{02}
\omega \left( A \right)\le \frac{1}{2}\left( \left\| A \right\|+{{\left\| {{A}^{2}} \right\|}^{\frac{1}{2}}} \right).
\end{equation}
He also showed the following estimate, which is stronger than \eqref{02},
\begin{equation}\label{07}
\omega \left( A \right)\le \frac{1}{2}\left\| \left| A \right|+\left| {{A}^{*}} \right| \right\|.
\end{equation}
Another refinement of the second inequality in \eqref{0} has been  established in \cite{6}. This refinement asserts that if $A\in \mathbb{B}\left( \mathscr{H} \right)$, then 
\begin{equation}\label{04}
{{\omega }^{2}}\left( A \right)\le \frac{1}{2}\left\| {{\left| A \right|}^{2}}+{{\left| {{A}^{*}} \right|}^{2}} \right\|.
\end{equation}
 Also, in the same paper, the author proved that
\begin{equation}\label{01}
\frac{1}{4}\left\| {{\left| A \right|}^{2}}+{{\left| {{A}^{*}} \right|}^{2}} \right\|\le {{\omega }^{2}}\left( A \right).
\end{equation}
 It can be easily seen that \eqref{01} improves the first inequality in \eqref{0}. It should be mentioned here that upper bounds obtained in \eqref{02} and \eqref{04} are not comparable.

Nowadays, a considerable attention is dedicated to refinement and generalization of the previous discussion \cite{m3, m10, m, m1, m2}.

Motivated by the results outlined above, it is the primary goal of the present work to establish new numerical radius inequalities for Hilbert space operators. A comparison of the established results with previously obtained results is demonstrated to show that the results presented in this paper are better than those already exist in literature.\\
In Section \ref{03}, we introduce an inequality that refines both inequalities \eqref{02} and \eqref{04}. Furthermore, we make a refinement of the inequality \eqref{01}. 

To achieve our goal, we need the following lemmas. The first  lemma presented by Kittaneh in \cite[(19)]{2}, while the second  lemma is given in \cite{05}.

\begin{lemma}\label{17}
Let $A,B\in \mathbb{B}\left( \mathscr{H} \right)$ be positive operators, then
\[\left\| A+B \right\|\le \frac{1}{2}\left( \left\| A \right\|+\left\| B \right\|+\sqrt{{{\left( \left\| A \right\|-\left\| B \right\| \right)}^{2}}+4{{\left\| {{A}^{\frac{1}{2}}}{{B}^{\frac{1}{2}}} \right\|}^{2}}} \right).\]
In particular,
\[\left\| {{\left| A \right|}^{2}}+{{\left| {{A}^{*}} \right|}^{2}} \right\|\le \left\| {{A}^{2}} \right\|+{{\left\| A \right\|}^{2}},\]
for any $A\in \mathbb{B}\left( \mathscr{H} \right)$.
\end{lemma}

\begin{lemma}\label{08}
Let $A\in \mathbb{B}\left( \mathscr{H} \right)$ and let $x,y\in \mathscr H$ be any vector. If $f,g$ are nonnegative continuous functions on $\left[ 0,\infty  \right)$ satisfying $f\left( t \right)g\left( t \right)=t,\left( t\ge 0 \right)$, then
\[\left| \left\langle Ax,y \right\rangle  \right|\le \left\| f\left( \left| A \right| \right)x \right\|\left\| g\left( \left| {{A}^{*}} \right| \right)y \right\|.\]
In particular,
\[\left| \left\langle Ax,y \right\rangle  \right|\le \sqrt{\left\langle {{\left| A \right|}^{2\left( 1-v \right)}}x,x \right\rangle \left\langle {{\left| {{A}^{*}} \right|}^{2v}}y,y \right\rangle },\text{ }\left( 0\le v\le 1 \right).\]
\end{lemma}

\section{Numerical Radius Inequalities}\label{03}
Using the same strategy as in \cite[Theorem 1]{el}, we get the first result, a refinement of inequality \eqref{07}.
\begin{theorem}\label{0000}
Let $A\in \mathbb{B}\left( \mathscr{H} \right)$, then
\begin{equation}\label{29}
\omega \left( A \right)\le \frac{1}{2}~\underset{0\le v\le 1}{\mathop{\min }}\,\left\| {{\left| A \right|}^{2\left( 1-v \right)}}+{{\left| {{A}^{*}} \right|}^{2v}} \right\|.
\end{equation}
\end{theorem}
\begin{proof}
Let $x\in \mathscr{H}$ be a unit vector. By employing Lemma \ref{08} and the arithmetic-geometric mean inequality, we have
	\[\begin{aligned}
   \left| \left\langle Ax,x \right\rangle  \right|&\le \sqrt{\left\langle {{\left| A \right|}^{2\left( 1-v \right)}}x,x \right\rangle \left\langle {{\left| {{A}^{*}} \right|}^{2v}}x,x \right\rangle } \\ 
 & \le \frac{1}{2}\left( \left\langle {{\left| A \right|}^{2\left( 1-v \right)}}x,x \right\rangle +\left\langle {{\left| {{A}^{*}} \right|}^{2v}}x,x \right\rangle  \right) \\ 
 & =\frac{1}{2}\left\langle \left( {{\left| A \right|}^{2\left( 1-v \right)}}+{{\left| {{A}^{*}} \right|}^{2v}} \right)x,x \right\rangle  \\ 
 & \le\frac{1}{2}\left\| {{\left| A \right|}^{2\left( 1-v \right)}}+{{\left| {{A}^{*}} \right|}^{2v}} \right\|.  
\end{aligned}\]
Thus,
	\[\left| \left\langle Ax,x \right\rangle  \right|\le \frac{1}{2}\left\| {{\left| A \right|}^{2\left( 1-v \right)}}+{{\left| {{A}^{*}} \right|}^{2v}} \right\|.\]
Now, taking the supremum over $x\in \mathscr H$ with $\left\| x \right\|=1$ in the above inequality produces
\begin{equation}\label{00}
\omega \left( A \right)\le \frac{1}{2}\left\| {{\left| A \right|}^{2\left( 1-v \right)}}+{{\left| {{A}^{*}} \right|}^{2v}} \right\|.
\end{equation}
Taking minimum over all $v\in \left[ 0,1 \right]$, we get 
	\[\omega \left( A \right)\le \frac{1}{2}~\underset{0\le v\le 1}{\mathop{\min }}\,\left\| {{\left| A \right|}^{2\left( 1-v \right)}}+{{\left| {{A}^{*}} \right|}^{2v}} \right\|,\]
as required.
\end{proof}

The following example shows that inequality \eqref{29} is a nontrivial improvement of inequality \eqref{07}. It shows that
\[\frac{1}{2}~\underset{0\le v\le 1}{\mathop{\min }}\,\left\| {{\left| A \right|}^{2\left( 1-v \right)}}+{{\left| {{A}^{*}} \right|}^{2v}} \right\|\lneqq \frac{1}{2}\left\| \left| A \right|+\left| {{A}^{*}} \right| \right\|.\]

\begin{example}
Letting $A=\left[ \begin{matrix}
   0 & 1 & 0  \\
   0 & 0 & 2  \\
   0 & 0 & 0  \\
\end{matrix} \right]$. So, $\left| A \right|=\left[ \begin{matrix}
   0 & 0 & 0  \\
   0 & 1 & 0  \\
   0 & 0 & 2  \\
\end{matrix} \right]$ and $\left| {{A}^{*}} \right|=\left[ \begin{matrix}
   1 & 0 & 0  \\
   0 & 2 & 0  \\
   0 & 0 & 0  \\
\end{matrix} \right]$. With a help of Matlab, we get
	\[\frac{1}{2}\left\| \left| A \right|+\left| {{A}^{*}} \right| \right\|=\frac{3}{2},\]
while
	\[\frac{1}{2}~\underset{0\le v\le 1}{\mathop{\min }}\,\left\| {{\left| A \right|}^{2\left( 1-v \right)}}+{{\left| {{A}^{*}} \right|}^{2v}} \right\|\approx1.28.\]
\end{example}

Theorem \ref{0000} admits the following  result.
\begin{corollary}
Let $A\in \mathbb{B}\left( \mathscr{H} \right)$ and let $0\le v\le 1$. Then
\[\omega \left( A \right)\le \frac{1}{4}\left( {{\left\| A \right\|}^{2\left( 1-v \right)}}+{{\left\| A \right\|}^{2v}}+\sqrt{{{\left( {{\left\| A \right\|}^{2\left( 1-v \right)}}-{{\left\| A \right\|}^{2v}} \right)}^{2}}+4{{\left\| {{\left| A \right|}^{1-v}}{{\left| {{A}^{*}} \right|}^{v}} \right\|}^{2}}} \right),\]
\end{corollary}
\begin{proof}
Let $0\le v\le 1$. We have
\[\begin{aligned}
  & \left\| {{\left| A \right|}^{2\left( 1-v \right)}}+{{\left| {{A}^{*}} \right|}^{2v}} \right\| \\ 
 & \le \frac{1}{2}\left( \left\| {{\left| A \right|}^{2\left( 1-v \right)}} \right\|+\left\| {{\left| {{A}^{*}} \right|}^{2v}} \right\|+\sqrt{{{\left( \left\| {{\left| A \right|}^{2\left( 1-v \right)}} \right\|-\left\| {{\left| {{A}^{*}} \right|}^{2v}} \right\| \right)}^{2}}+4{{\left\| {{\left| A \right|}^{1-v}}{{\left| {{A}^{*}} \right|}^{v}} \right\|}^{2}}} \right) \\ 
 &\qquad \text{(by Lemma \ref{17})}\\
 & =\frac{1}{2}\left( {{\left\| A \right\|}^{2\left( 1-v \right)}}+{{\left\| A \right\|}^{2v}}+\sqrt{{{\left( {{\left\| A \right\|}^{2\left( 1-v \right)}}-{{\left\| A \right\|}^{2v}} \right)}^{2}}+4{{\left\| {{\left| A \right|}^{1-v}}{{\left| {{A}^{*}} \right|}^{v}} \right\|}^{2}}} \right) .
 \end{aligned}\]
So, it follows from the inequality \eqref{00} that
\[\omega \left( A \right)\le \frac{1}{4}\left( {{\left\| A \right\|}^{2\left( 1-v \right)}}+{{\left\| A \right\|}^{2v}}+\sqrt{{{\left( {{\left\| A \right\|}^{2\left( 1-v \right)}}-{{\left\| A \right\|}^{2v}} \right)}^{2}}+4{{\left\| {{\left| A \right|}^{1-v}}{{\left| {{A}^{*}} \right|}^{v}} \right\|}^{2}}} \right),\]
as required.
\end{proof}

\begin{remark}
From \cite[Theorem IX.2.1]{bhatia} we know that, if $A,B\in \mathbb{B}\left( \mathscr{H} \right)$ are positive operators, then
	\[\left\| {{A}^{p}}{{B}^{p}} \right\|\le {{\left\| AB \right\|}^{p}},\text{ }\left( 0\le p\le 1 \right).\]
Based on this inequality, one can easily infer that
\[\begin{aligned}
   \omega \left( A \right)&\le \frac{1}{2}\left( \left\| A \right\|+\left\| {{\left| A \right|}^{\frac{1}{2}}}{{\left| {{A}^{*}} \right|}^{\frac{1}{2}}} \right\| \right) \\ 
 & \le \frac{1}{2}\left( \left\| A \right\|+{{\left\| \left| A \right|\left| {{A}^{*}} \right| \right\|}^{\frac{1}{2}}} \right) \\ 
 & = \frac{1}{2}\left( \left\| A \right\|+{{\left\| {{A}^{2}} \right\|}^{\frac{1}{2}}} \right).  
\end{aligned}\]
\end{remark}

Our second result reads as follows. This result contains our promised refinement of the inequality \eqref{04}. 
\begin{theorem}\label{16}
Let $A\in \mathbb{B}\left( \mathscr{H} \right)$ and let $f,g$ be nonnegative continuous functions on $\left[ 0,\infty  \right)$ satisfying $f\left( t \right)g\left( t \right)=t,\left( t\ge 0 \right)$, then
\begin{equation}\label{8}
\begin{aligned}
   {{\omega }^{2}}\left( A \right)&\le \frac{1}{4}\left\| {{f}^{4}}\left( \left| A \right| \right)+{{g}^{4}}\left( \left| {{A}^{*}} \right| \right) \right\|+\frac{1}{4}\left\| {{f}^{2}}\left( \left| A \right| \right){{g}^{2}}\left( \left| {{A}^{*}} \right| \right)+{{g}^{2}}\left( \left| {{A}^{*}} \right| \right){{f}^{2}}\left( \left| A \right| \right) \right\| \\ 
 & \le \frac{1}{2}\left\| {{f}^{4}}\left( \left| A \right| \right)+{{g}^{4}}\left( \left| {{A}^{*}} \right| \right) \right\|.  
\end{aligned}
\end{equation}
\end{theorem}
\begin{proof}
On making use of Lemma \ref{08}, we have
\[\begin{aligned}
  & {{\left| \left\langle Ax,x \right\rangle  \right|}^{2}} \\ 
 & \le \left\langle {{f}^{2}}\left( \left| A \right| \right)x,x \right\rangle \left\langle {{g}^{2}}\left( \left| {{A}^{*}} \right| \right)x,x \right\rangle  \\ 
 & \le {{\left( \frac{\left\langle {{f}^{2}}\left( \left| A \right| \right)x,x \right\rangle +\left\langle {{g}^{2}}\left( \left| {{A}^{*}} \right| \right)x,x \right\rangle }{2} \right)}^{2}} \\ 
 &\qquad \text{(by the arithmetic-geometric mean inequality)} \\
 & \le \frac{1}{4}{{\left\langle \left( {{f}^{2}}\left( \left| A \right| \right)+{{g}^{2}}\left( \left| {{A}^{*}} \right| \right) \right)x,x \right\rangle }^{2}} \\ 
 & \le \frac{1}{4}\left\langle {{\left( {{f}^{2}}\left( \left| A \right| \right)+{{g}^{2}}\left( \left| {{A}^{*}} \right| \right) \right)}^{2}}x,x \right\rangle  \\ 
  &\qquad \text{(by the Cauchy-Schwarz inequality)} \\
 & = \frac{1}{4}\left\langle \left( {{f}^{4}}\left( \left| A \right| \right)+{{g}^{4}}\left( \left| {{A}^{*}} \right| \right)+{{f}^{2}}\left( \left| A \right| \right){{g}^{2}}\left( \left| {{A}^{*}} \right| \right)+{{g}^{2}}\left( \left| {{A}^{*}} \right| \right){{f}^{2}}\left( \left| A \right| \right) \right)x,x \right\rangle  \\ 
 & \le \frac{1}{4}\left\| {{f}^{4}}\left( \left| A \right| \right)+{{g}^{4}}\left( \left| {{A}^{*}} \right| \right)+{{f}^{2}}\left( \left| A \right| \right){{g}^{2}}\left( \left| {{A}^{*}} \right| \right)+{{g}^{2}}\left( \left| {{A}^{*}} \right| \right){{f}^{2}}\left( \left| A \right| \right) \right\| \\ 
 & \le \frac{1}{4}\left\| {{f}^{4}}\left( \left| A \right| \right)+{{g}^{4}}\left( \left| {{A}^{*}} \right| \right) \right\|+\frac{1}{4}\left\| {{f}^{2}}\left( \left| A \right| \right){{g}^{2}}\left( \left| {{A}^{*}} \right| \right)+{{g}^{2}}\left( \left| {{A}^{*}} \right| \right){{f}^{2}}\left( \left| A \right| \right) \right\|  
   \\&\qquad \text{(by the triangle inequality for the usual operator norm)} .
\end{aligned}\]
Thus,
\[{{\left| \left\langle Ax,x \right\rangle  \right|}^{2}}\le \frac{1}{4}\left\| {{f}^{4}}\left( \left| A \right| \right)+{{g}^{4}}\left( \left| {{A}^{*}} \right| \right) \right\|+\frac{1}{4}\left\| {{f}^{2}}\left( \left| A \right| \right){{g}^{2}}\left( \left| {{A}^{*}} \right| \right)+{{g}^{2}}\left( \left| {{A}^{*}} \right| \right){{f}^{2}}\left( \left| A \right| \right) \right\|.\]
Now, taking the supremum over $x\in \mathscr H$ with $\left\| x \right\|=1$ in the above inequality produces
\begin{equation}\label{12}
{{\omega }^{2}}\left( A \right)\le \frac{1}{4}\left\| {{f}^{4}}\left( \left| A \right| \right)+{{g}^{4}}\left( \left| {{A}^{*}} \right| \right) \right\|+\frac{1}{4}\left\| {{f}^{2}}\left( \left| A \right| \right){{g}^{2}}\left( \left| {{A}^{*}} \right| \right)+{{g}^{2}}\left( \left| {{A}^{*}} \right| \right){{f}^{2}}\left( \left| A \right| \right) \right\|.
\end{equation}
On the other hand,
\begin{equation}\label{18}
\begin{aligned}
  & \left\| {{f}^{2}}\left( \left| A \right| \right){{g}^{2}}\left( \left| {{A}^{*}} \right| \right)+{{g}^{2}}\left( \left| {{A}^{*}} \right| \right){{f}^{2}}\left( \left| A \right| \right) \right\| \\ 
 & =\frac{1}{2}\left\| {{\left( {{f}^{2}}\left( \left| A \right| \right)+{{g}^{2}}\left( \left| {{A}^{*}} \right| \right) \right)}^{2}}-{{\left( {{f}^{2}}\left( \left| A \right| \right)-{{g}^{2}}\left( \left| {{A}^{*}} \right| \right) \right)}^{2}} \right\| \\ 
 & \le \frac{1}{2}\left\| {{\left( {{f}^{2}}\left( \left| A \right| \right)+{{g}^{2}}\left( \left| {{A}^{*}} \right| \right) \right)}^{2}}+{{\left( {{f}^{2}}\left( \left| A \right| \right)-{{g}^{2}}\left( \left| {{A}^{*}} \right| \right) \right)}^{2}} \right\| \\ 
 & =\left\| {{f}^{4}}\left( \left| A \right| \right)+{{g}^{4}}\left( \left| {{A}^{*}} \right| \right) \right\| \\ 
\end{aligned}
\end{equation}
in which the last inequality is a direct consequence of the following result (see, e.g.,  \cite[Corollary 1]{2}):
\begin{equation}\label{14}
\left\| S-T \right\|\le \left\| S+T \right\|,
\end{equation}
where $S,T\in \mathbb{B}\left( \mathscr{H} \right)$ are two positive operators. Combining the relations \eqref{12} and \eqref{18}, we get \eqref{8}.
\end{proof}

\begin{remark}
If we take $f\left( t \right)={{t}^{1-v}}$ and $g\left( t \right)={{t}^{v}}$ with $0\le v\le 1$, in the inequality \eqref{8}, then we get
\[\begin{aligned}
   {{\omega }^{2}}\left( A \right)&\le \frac{1}{4}\left\| {{\left| A \right|}^{4\left( 1-v \right)}}+{{\left| {{A}^{*}} \right|}^{4v}} \right\|+\frac{1}{4}\left\| {{\left| A \right|}^{2\left( 1-v \right)}}{{\left| {{A}^{*}} \right|}^{2v}}+{{\left| {{A}^{*}} \right|}^{2v}}{{\left| A \right|}^{2\left( 1-v \right)}} \right\| \\ 
 & \le \frac{1}{2}\left\| {{\left| A \right|}^{4\left( 1-v \right)}}+{{\left| {{A}^{*}} \right|}^{4v}} \right\|.  
\end{aligned}\]
In particular,
\[{{\omega }^{2}}\left( A \right)\le \frac{1}{4}\left\| {{\left| A \right|}^{2}}+{{\left| {{A}^{*}} \right|}^{2}} \right\|+\frac{1}{4}\left\| \left| A \right|\left| {{A}^{*}} \right|+\left| {{A}^{*}} \right|\left| A \right| \right\|\le \frac{1}{2}\left\| {{\left| A \right|}^{2}}+{{\left| {{A}^{*}} \right|}^{2}} \right\|.\]
\end{remark}

\begin{remark}
For normal operator $A$, we get on both sides of 
\begin{equation}\label{25}
{{\omega }^{2}}\left( A \right)\le \frac{1}{4}\left( \left\| {{\left| A \right|}^{2}}+{{\left| {{A}^{*}} \right|}^{2}} \right\|+\left\| \left| A \right|\left| {{A}^{*}} \right|+\left| {{A}^{*}} \right|\left| A \right| \right\| \right),
\end{equation}
 the same quantity ${{\left\| A \right\|}^{2}}$, which shows that the constant ${1}/{4}\;$
is best possible in general in the inequality \eqref{25}.
\end{remark}

The following corollary shows that inequality \eqref{25} is also
sharper than the inequality \eqref{02}.
\begin{corollary}
Let $A\in \mathbb{B}\left( \mathscr{H} \right)$, then
\begin{equation}\label{26}
\omega \left( A \right)\le \frac{1}{2}\sqrt{\left\| {{\left| A \right|}^{2}}+{{\left| {{A}^{*}} \right|}^{2}} \right\|+\left\| \left| A \right|\left| {{A}^{*}} \right|+\left| {{A}^{*}} \right|\left| A \right| \right\|}\le \frac{1}{2}\left( \left\| A \right\|+{{\left\| {{A}^{2}} \right\|}^{\frac{1}{2}}} \right).
\end{equation}
\end{corollary}
\begin{proof}
Observe that
\begin{equation}\label{15}
\begin{aligned}
   &\left\| \left| A \right|\left| {{A}^{*}} \right|+\left| {{A}^{*}} \right|\left| A \right| \right\|\\
  & \le \left\| \left| A \right|\left| {{A}^{*}} \right| \right\|+\left\| \left| {{A}^{*}} \right|\left| A \right| \right\| \quad \text{(by the triangle inequality for the usual operator norm)}\\ 
 & =\left\| \left| A \right|\left| {{A}^{*}} \right| \right\|+\left\| {{\left( \left| A \right|\left| {{A}^{*}} \right| \right)}^{*}} \right\| \\ 
 & =2\left\| \left| A \right|\left| {{A}^{*}} \right| \right\| \quad \text{(since $\left\| T \right\|=\left\| {{T}^{*}} \right\|$ for any $T\in \mathbb{B}\left( \mathscr{H} \right)$)}\\ 
 & =2\left\| {{A}^{2}} \right\|\quad \text{(since $\left\| \left| A \right|\left| {{A}^{*}} \right| \right\|=\left\| {{A}^{2}} \right\|$)}.  
\end{aligned}
\end{equation}
On the other hand,
\[\begin{aligned}
   {{\omega }^{2}}\left( A \right)&\le \frac{1}{4}\left( \left\| {{\left| A \right|}^{2}}+{{\left| {{A}^{*}} \right|}^{2}} \right\|+\left\| \left| A \right|\left| {{A}^{*}} \right|+\left| {{A}^{*}} \right|\left| A \right| \right\| \right) \\ 
 & \le \frac{1}{4}\left( \left\| {{A}^{2}} \right\|+{{\left\| A \right\|}^{2}} \right)+\frac{1}{4}\left\| \left| A \right|\left| {{A}^{*}} \right|+\left| {{A}^{*}} \right|\left| A \right| \right\| \quad \text{(by Lemma \ref{17})}\\ 
 & \le\frac{1}{4}\left( \left\| {{A}^{2}} \right\|+{{\left\| A \right\|}^{2}} \right)+\frac{1}{2}\left\| {{A}^{2}} \right\| \quad \text{(by \eqref{15})}\\ 
 & =\frac{1}{4}\left( {{\left\| A \right\|}^{2}}+3\left\| {{A}^{2}} \right\| \right) \\ 
  & \le \frac{1}{4}\left( {{\left\| A \right\|}^{2}}+2\left\| A \right\|{{\left\| {{A}^{2}} \right\|}^{\frac{1}{2}}}+\left\| {{A}^{2}} \right\| \right) \quad \text{(since $\left\| {{A}^{2}} \right\|={{\left\| {{A}^{2}} \right\|}^{\frac{1}{2}}}{{\left\| {{A}^{2}} \right\|}^{\frac{1}{2}}}\le \left\| A \right\|{{\left\| {{A}^{2}} \right\|}^{\frac{1}{2}}}$)}\\ 
 & =\frac{1}{4}{{\left( \left\| A \right\|+{{\left\| {{A}^{2}} \right\|}^{\frac{1}{2}}} \right)}^{2}},  
\end{aligned}\]
and yields validity of inequality \eqref{26}.
\end{proof}

\begin{remark}
It follows from the work in \cite{satary} that
\begin{equation}\label{24}
{{\omega }^{2}}\left( A \right)\le \frac{1}{4}\left\| {{\left| A \right|}^{2}}+{{\left| {{A}^{*}} \right|}^{2}} \right\|+\frac{1}{2}\omega \left( \left| A \right| \left| A^{*} \right| \right).
\end{equation}
On the other hand, from the inequality \eqref{25}, we get
\[\begin{aligned}
   {{\omega }^{2}}\left( A \right)&\le \frac{1}{4}\left( \left\| {{\left| A \right|}^{2}}+{{\left| {{A}^{*}} \right|}^{2}} \right\|+\left\| \left| A \right|\left| {{A}^{*}} \right|+\left| {{A}^{*}} \right|\left| A \right| \right\| \right) \\ 
&=\frac{1}{4}\left( \left\| {{\left| A \right|}^{2}}+{{\left| {{A}^{*}} \right|}^{2}} \right\|+\omega\left( \left| A \right|\left| {{A}^{*}} \right|+\left| {{A}^{*}} \right|\left| A \right| \right) \right) \\ 
&\qquad \text{(since $\left| A \right| \left| A^{*} \right| + \left| A^{*} \right| \left| A \right|$ is self-adjoint)}\\
& \le \frac{1}{4}\left( \left\| {{\left| A \right|}^{2}}+{{\left| {{A}^{*}} \right|}^{2}} \right\|+\omega\left( \left| A \right|\left| {{A}^{*}} \right|\right)+\omega\left(\left| {{A}^{*}} \right|\left| A \right| \right) \right)\\
 &\qquad \text{(by the triangle inequality for $\omega\left( \cdot \right)$)}\\ 
& = \frac{1}{4}\left( \left\| {{\left| A \right|}^{2}}+{{\left| {{A}^{*}} \right|}^{2}} \right\|+\omega\left( \left| A \right|\left| {{A}^{*}} \right|\right)+\omega\left(\left(\left| {{A}} \right|\left| A^{*} \right| \right)^{*} \right)\right) \\
& = \frac{1}{4}\left\| {{\left| A \right|}^{2}}+{{\left| {{A}^{*}} \right|}^{2}} \right\|+\frac{1}{2}\omega\left( \left| A \right|\left| {{A}^{*}} \right|\right)\quad \text{(since $\omega\left( T\right) = \omega\left( T^{*} \right)$ for any $T \in \mathbb{B}(\mathscr{H})$)}.
\end{aligned}\]
 It is clear that  inequality \eqref{25} refines the inequality \eqref{24}.
\end{remark}

\begin{remark}
From the inequality \eqref{15} and the first inequality in \eqref{0}, we infer that
	\[\frac{1}{4}\left\| \left| A \right|\left| {{A}^{*}} \right|+\left| {{A}^{*}} \right|\left| A \right| \right\|\le \frac{1}{2}\left\| {{A}^{2}} \right\|\le \omega \left( {{A}^{2}} \right).\]
Therefore, by inequality \eqref{25}, we can write
\[{{\omega }^{2}}\left( A \right)\le \frac{1}{4}\left\| {{\left| A \right|}^{2}}+{{\left| {{A}^{*}} \right|}^{2}} \right\|+\omega \left( {{A}^{2}} \right).\]
This can be regarded as a reverse of the inequality \eqref{06}, when $n=2$.
\end{remark}

We can obtain a refinement of the triangle inequality as follows.
\begin{proposition}\label{1}
Let $A,B\in \mathbb{B}\left( \mathscr{H} \right)$, then
\begin{equation}\label{6}
\begin{aligned}
   \left\| A+B \right\|&\le \sqrt{\left\| {{\left| A \right|}^{2}}+{{\left| B \right|}^{2}} \right\|+\left\| {{A}^{*}}B+{{B}^{*}}A \right\|} \\ 
 & \le \sqrt{{{\left\| A \right\|}^{2}}+{{\left\| B \right\|}^{2}}+2\left\| {{A}^{*}}B \right\|} \\ 
 & \le \left\| A \right\|+\left\| B \right\|.  
\end{aligned}
\end{equation}
\end{proposition}
\begin{proof}
We can write,
\begin{align}
   {{\left\| A+B \right\|}^{2}}&=\left\| {{\left| A \right|}^{2}}+{{\left| B \right|}^{2}}+{{A}^{*}}B+{{B}^{*}}A \right\| \nonumber\\ 
 & \le \left\| {{\left| A \right|}^{2}}+{{\left| B \right|}^{2}} \right\|+\left\| {{A}^{*}}B+{{B}^{*}}A \right\| \nonumber\\
 &\qquad \text{(by the triangle inequality for the usual operator norm)}\label{23}\\ 
 & \le {{\left\| A \right\|}^{2}}+{{\left\| B \right\|}^{2}}+\left\| {{A}^{*}}B \right\|+\left\| {{B}^{*}}A \right\| \nonumber\\ 
 &\qquad \text{(by the triangle inequality for the usual operator norm)}\nonumber\\
 & ={{\left\| A \right\|}^{2}}+{{\left\| B \right\|}^{2}}+\left\| {{A}^{*}}B \right\|+\left\| {{\left( {{A}^{*}}B \right)}^{*}} \right\| \nonumber\\ 
 & ={{\left\| A \right\|}^{2}}+{{\left\| B \right\|}^{2}}+2\left\| {{A}^{*}}B \right\| \quad \text{(since $\left\| T \right\|=\left\| {{T}^{*}} \right\|$ for any $T\in \mathbb{B}\left( \mathscr{H} \right)$)}\nonumber\\ 
 & \le {{\left\| A \right\|}^{2}}+{{\left\| B \right\|}^{2}}+2\left\| {{A}^{*}} \right\|\left\| B \right\|\nonumber\\
 & \qquad \text{(by the submultiplicative property of usual operator norm)}\nonumber\\ 
 & ={{\left\| A \right\|}^{2}}+{{\left\| B \right\|}^{2}}+2\left\| A \right\|\left\| B \right\| \nonumber\\ 
 & ={{\left( \left\| A \right\|+\left\| B \right\| \right)}^{2}}.  \nonumber
\end{align}
This shows that
\[\begin{aligned}
   {{\left\| A+B \right\|}^{2}}&\le \left\| {{\left| A \right|}^{2}}+{{\left| B \right|}^{2}} \right\|+\left\| {{A}^{*}}B+{{B}^{*}}A \right\| \\ 
 & \le {{\left\| A \right\|}^{2}}+{{\left\| B \right\|}^{2}}+2\left\| {{A}^{*}}B \right\| \\ 
 & \le {{\left( \left\| A \right\|+\left\| B \right\| \right)}^{2}}.  
\end{aligned}\]
Taking the square root in this inequality we obtain \eqref{6}.
\end{proof}

Our refinement of the inequality \eqref{01} is presented in the following theorem.
\begin{theorem}
Let $A\in \mathbb{B}\left( \mathscr{H} \right)$, then
\[\frac{1}{4}\left\| {{\left| A \right|}^{2}}+\left| {{A}^{*}} \right|^2 \right\|\le \frac{1}{2}\sqrt{2{{\omega }^{4}}\left( A \right)+\frac{1}{8}\left\| {{\left( A+{{A}^{*}} \right)}^{2}}{{\left( A-{{A}^{*}} \right)}^{2}} \right\|}\le {{\omega }^{2}}\left( A \right).\]
\end{theorem}
\begin{proof}
Let $A=B+iC$ be the Cartesian decomposition of $A$. Then $B$
and $C$ are self-adjoint and
\[{{\left| \left\langle Ax,x \right\rangle  \right|}^{2}}={{\left\langle Bx,x \right\rangle }^{2}}+{{\left\langle Cx,x \right\rangle }^{2}}.\]
A little calculation shows that
\begin{equation}\label{27}
\left\| B \right\|\le \omega \left( A \right),
\end{equation}
and similarly
\begin{equation}\label{28}
\left\| C \right\|\le \omega \left( A \right).
\end{equation}
Now, by using Proposition \ref{1} we have
\[\begin{aligned}
   \frac{1}{4}\left\| {{\left| A \right|}^{2}}+\left| {{A}^{*}} \right|^2 \right\|&=\frac{1}{2}\left\| {{B}^{2}}+{{C}^{2}} \right\| \\ 
 & \le \frac{1}{2}\sqrt{{{\left\| B \right\|}^{4}}+{{\left\| C \right\|}^{4}}+2\left\| {{B}^{2}}{{C}^{2}} \right\|} \\
 &\qquad \text{(by the submultiplicative property of usual operator norm)} \\
 & \le \frac{1}{2}\sqrt{2{{\omega }^{4}}\left( A \right)+2\left\| {{B}^{2}}{{C}^{2}} \right\|} \quad \text{(by \eqref{27} and \eqref{28})}\\ 
 & \le \frac{1}{2}\sqrt{2{{\omega }^{4}}\left( A \right)+2{{\left\| B \right\|}^{2}}{{\left\| C \right\|}^{2}}} \\ 
 & \le {{\omega }^{2}}\left( A \right)  \quad \text{(by \eqref{27} and \eqref{28})}.
\end{aligned}\]
Thus, from the discussion above we have
\[\frac{1}{4}\left\| {{\left| A \right|}^{2}}+\left| {{A}^{*}} \right|^2 \right\|\le \frac{1}{2}\sqrt{2\left( {{\omega }^{4}}\left( A \right)+\left\| {{B}^{2}}{{C}^{2}} \right\| \right)}\le {{\omega }^{2}}\left( A \right).\]
This completes the proof.
\end{proof}

\section*{Acknowledgement}
This project was financially supported by Islamic Azad University, Mashhad Branch.

{\tiny \vskip 0.3 true cm } 

{\tiny (F. P. Najafabadi) Department of Mathematics, Mashhad Branch, Islamic Azad University, Mashhad, Iran.}

{\tiny \textit{E-mail address:}  	jamatia.math@gmail.com}

{\tiny \vskip 0.3 true cm } 

{\tiny (H. R. Moradi) Department of Mathematics, Mashhad Branch, Islamic Azad University, Mashhad, Iran.}

{\tiny \textit{E-mail address:} hrmoradi@mshdiau.ac.ir }

\end{document}